\documentclass[12pt, a4paper]{article}
\usepackage{amsfonts}
\usepackage{mathrsfs}
\usepackage{latexsym}
\usepackage{xy}
\usepackage{amsfonts,amsmath,amssymb,amsthm}
\usepackage{color}

\xyoption{all}

\newcommand{\bcen}{\begin{center}}     \newcommand{\ecen}{\end{center}}
\newcommand{\bay}{\begin{array}}      \newcommand{\eay}{\end{array}}
\newcommand{\beq}{\begin{eqnarray*}}      \newcommand{\eeq}{\end{eqnarray*}}
\def\gl{\mathrm{gl.dim}}

\def\rad{\mathrm{rad}}

\def\sup{\mathrm{sup}}

\def\Hom{\mathrm{Hom}}
\def\Ker{\mathrm{Ker}}
\def\Im{\mathrm{Im}}
\def\Ext{\mathrm{Ext}}
\def\Tor{\mathrm{Tor}}

\def\mod{\mathrm{mod}}
\def\Mod{\mathrm{Mod}}
\def\id{\mathrm{id}}
\def\pd{\mathrm{pd}}

\def\per{\mathrm{per}}
\def\proj{\mathrm{proj}}
\def\Gproj{\mathrm{Gproj}}

\def\inj{\mathrm{inj}}

\def\Cone{\mathrm{Cone}}

\def\thick{\mathrm{thick}}

\begin{document}

\newtheorem{theorem}{Theorem}[section]
\newtheorem{proposition}[theorem]{Proposition}
\newtheorem{lemma}[theorem]{Lemma}
\newtheorem{corollary}[theorem]{Corollary}
\newtheorem{remark}[theorem]{Remark}
\newtheorem{example}[theorem]{Example}
\newtheorem{definition}[theorem]{Definition}
\newtheorem{question}[theorem]{Question}
\numberwithin{equation}{section}

\title{\large\bf
Eventually homological isomorphisms and Gorenstein projective
modules}

\author{\large Yongyun Qin}

\date{\footnotesize College of Mathematics and Statistics,
Qujing Normal University, \\ Qujing, Yunnan 655011, China. E-mail:
qinyongyun2006@126.com}

\maketitle

\begin{abstract} We prove that a certain eventually homological isomorphism
between module categories induces a triangle equivalence between their
singularity categories, Gorenstein defect categories and the stable categories of Gorenstein
projective modules. Further, we show that Auslander-Reiten conjecture
and Gorenstein symmetry conjecture
can be reduced by eventually homological isomorphisms.
Applying the results to arrow removal and vertex removal, we describe the Gorenstein projective
modules over some non-monomial algebras, and we verify the Auslander-Reiten conjecture for certain
algebras.
\end{abstract}

\medskip

{\footnotesize {\bf Mathematics Subject Classification (2020)}:
16E35; 16E65; 16G50; 18G25; 18G80.}

\medskip

{\footnotesize {\bf Keywords}: Singularity categories;
Gorenstein defect categories; Gorenstein projective
modules; Auslander-Reiten conjecture; Gorenstein symmetry conjecture;
Eventually homological isomorphisms. }

\bigskip

\section{\large Introduction}

\indent\indent
A functor $F:\mathcal{B}\rightarrow \mathcal{C}$ between abelian categories
is called an {\it eventually homological isomorphism}
if there is an integer $t$ such that
for every $j > t$, there is an
isomorphism $\Ext ^j _{\mathcal{B}}(X, Y ) \cong \Ext ^j _{\mathcal{C}}(FX, FY )$
for all objects $X, Y\in \mathcal{B}$.
This notion was introduced in \cite{PSS14}, and it arises naturally in
reducing homological properties of finite dimensional algebras.
Indeed, Psaroudakis et al. \cite{PSS14} characterized when the
functor $eA\otimes _A-: \mod A\rightarrow
\mod eAe$ is an eventually homological isomorphism,
and in that case, they transferred the Gorensteinness, singularity categories
and Fg condition of $A$ to $eAe$,
where $A$ is an algebra and $e$ is an idempotent of $A$.
Recently, Erdmann et al. showed that the arrow removal operation, passing from
a bound quiver algebra $A=kQ/I$ to $A/\langle \overline{\alpha} \rangle$,
yields an eventually homological isomorphism
if $\alpha$ is an arrow which does not occur in a minimal generating set
of $I$ \cite{EPS21}. We refer to \cite{GPS18, Qin20, WW20} for more discussions on eventually homological isomorphisms.

Recall that the {\it singularity category} $D_{sg}(A)$ of an algebra
$A$ is the Verdier quotient of the bounded derived category of finitely generated modules over $A$ by the
full subcategory of perfect complexes \cite{Buch87}.
According to \cite{Buch87}, there is an embedding functor
$F$ from the stable category $\underline{\Gproj} A$ of finitely generated Gorenstein projective
modules to $D_{sg}(A)$, and the {\it Gorenstein
defect category} of $A$ is defined to be Verdier quotient
$D_{def}(A):=D_{sg}(A)/ \Im F$, see \cite{BJO15}. This category
measures how far the algebra $A$ is from
being Gorenstein, because
$A$ is Gorenstein if and only if $D_{def}(A)$ is trivial \cite{BJO15}.
However, for non-Gorenstein algebras, not much is known
about their
Gorenstein defect categories. In recent years, many people described the Gorenstein projective
modules over some special kinds of algebras, such as Nakayama algebras \cite{Rin13} and monomial algebras \cite{CSZ18},
and some experts compared the Gorenstein defect categories
between two algebras related to one another \cite{CR20, LHZ22, Lu17, Lu19}.

In \cite{PSS14}, the authors proved that the Gorensteinness
of algebras is invariant under
certain eventually homological isomorphisms.
Inspired by this, we consider the following natural question:
is the Gorenstein
defect categories preserved under
eventually homological isomorphisms? and how about the homological conjectures
related to the Gorensteinness? We answer
these questions by the following theorem, which
is listed as Theorem~\ref{theorem-GSC}, Theorem~\ref{theorem-sing-equiv},
Theorem~\ref{theorem-Gproj-equiv}
and Theorem~\ref{theorem-ARC} in this paper.

\medskip

{\bf Theorem I.}  {\it Let $A$ and $B$ be two finite dimensional algebras,
and $F:\mod A\rightarrow \mod B$
be an eventually homological isomorphism
which is essentially surjective. Assume that $F$ admits a left adjoint
and a right adjoint. Then $F$
induces the following triangle equivalences
$$ D_{sg}(A)\cong D_{sg}(B), \ \underline{\Gproj }A\cong \underline{\Gproj }B
\ \mbox{and} \ D_{def}(A)\cong D_{def}(B),$$
and $A$ satisfies the Gorenstein
symmetry conjecture (resp. Auslander-Reiten conjecture,
Gorenstein projective conjecture) if and only if so does $B$.}

\medskip

Theorem I can be applied to arrow removal and vertex removal to reduce
the Gorenstein homological properties of algebras.

\medskip

{\bf Corollary I.} (Corollary~\ref{cor-arrow})
{\it Let $A = kQ/I$ be a quotient of a path algebra $kQ$
over a field $k$. Choose an arrow $\alpha$ in $Q$ such that $\alpha$ does not occur in a minimal generating set
of $I$ and define $B =A/\langle \overline{\alpha} \rangle$.
Then
$ \underline{\Gproj }A\cong \underline{\Gproj }B$ and
$D_{def}(A)\cong D_{def}(B).$ Moreover, $A$ satisfies the Gorenstein
symmetry conjecture (resp. Auslander-Reiten conjecture,
Gorenstein projective conjecture) if and only if so does $B$.
}

We mention that the above arrow removal operation was
used to reduce the Gorensteinness, singularity categories
and the Fg condition in \cite{EPS21}, and was investigated
with respect to the Hochschild (co)homology
and the finitistic dimension conjecture \cite{CLMS20, GPS18}.

Combining Theorem I with the result of Psaroudakis et al.,
we get the following corollary, which transfers the Gorenstein homological properties of $A$ to $eAe$.
This transition is called vertex removal in
\cite{FS92, GPS18}.
\medskip

{\bf Corollary II.} (Corollary~\ref{cor-idem})
{\it Let A be an algebra and $e$ be an idempotent in $A$. Assume that
$\pd _{eAe}eA< \infty $ and $\id _A(\frac {A/AeA}{\rad (A/AeA)})<\infty$
(or equivalently, $\pd _{(eAe)^{op}}Ae< \infty $ and $\pd _A(\frac {A/AeA}{\rad (A/AeA)})<\infty$).
Then the functor $eA\otimes _A-: \mod A\rightarrow
\mod eAe$ induces triangle equivalences $\underline{\Gproj }A\cong \underline{\Gproj }(eAe)$ and
$D_{def}(A)\cong D_{def}(eAe).$ Moreover, $A$ satisfies the Gorenstein
symmetry conjecture (resp. Auslander-Reiten conjecture,
Gorenstein projective conjecture) if and only if so does $eAe$ .
}

\medskip

We mention that the functor $eA\otimes _A-$
also induces a singular equivalence under the condition of Corollary II, see \cite{Chen09, PSS14}.
Now Corollary II can be compared with a recent result of Li et al. \cite{LHZ22}.
Assume that $\Tor ^{eAe}_i (Ae,G) = 0$ for any
$G \in \Gproj (eAe)$ and $i$ sufficiently large.
Then the functor $eA\otimes _A-$
induces triangle equivalences $D_{sg}(A)\cong D_{sg}(eAe)$, $D_{def}(A)\cong D_{def}(eAe)$
and $\underline{\Gproj }A\cong \underline{\Gproj }(eAe)$
if and only if $\pd _A(\frac {A/AeA}{\rad (A/AeA)})<\infty$, $\pd _{eAe}eA< \infty $
and both $eA\otimes _A-$ and $Ae\otimes _{eAe}-$ preserve modules of
finite Gorenstein projective dimension, see \cite[Corollary 1.2]{LHZ22}.
Here, we give a sufficient condition which seems more acceptable
to achieve these equivalences. Indeed, according to our proof,
the assumptions $\pd _{eAe}eA< \infty $ and $\id _A(\frac {A/AeA}{\rad (A/AeA)})<\infty$
infer that both $eA\otimes _A-$ and $Ae\otimes _{eAe}-$ preserve modules of
finite Gorenstein projective dimension.

As an application of
Corollary II, we reduce the Gorenstein homological
properties of a triangular matrix algebra to its corner algebras
under certain conditions, see Corollary~\ref{cor-tri-alg-A}
and Corollary~\ref{cor-tri-alg-B}. This reduction was also investigated
in \cite{LHZ22} with some different assumptions.
We give some concrete examples illustrating that
our results can be applied to study the Gorenstein projective
modules and the Auslander-Reiten conjecture
for some non-monomial algebras.

The paper is organized as follows. In section 2, we recall some relevant
definitions and conventions.
In section 3 we prove Theorem I.
In section 4, Corollary I
and Corollary II are proved, and the applications
on triangular matrix algebras and concrete examples are given.

 \section{\large Definitions and conventions}\label{Section-definitions and conventions}

\indent\indent In this section we will fix our notation and recall some basic definitions.

Let $\mathcal{S}$ be a set of objects of
a triangulated category $\mathcal{T}$. We denote by $\thick \mathcal{S}$
the smallest triangulated subcategory of $\mathcal{T}$
containing $\mathcal{S}$ and closed under taking direct summands.

Throughout $k$ is a fixed field and $D:=\Hom _k(-,k)$.
All algebras are assumed to be finite dimensional
associative $k$-algebras with identity
unless stated otherwise.
Let $A$ be such an algebra. We denote by $\Mod A$ the
category of left $A$-modules, and we view right $A$-modules
as left $A^{op}$-modules, where $A^{op}$ is the opposite algebra of $A$.
Denote by $\mod A$, $\proj A$
and $\inj A$ the full subcategories of $\Mod A$ consisting of all finitely
generated modules, finitely
generated projective modules and finitely
generated injective modules, respectively.
Let
$K^b (\proj A)$ (resp. $K^b (\inj A)$) be the bounded homotopy category of complexes
over $\proj A$ (resp. $\inj A$).
Let $\mathcal{D}(\Mod A)$ (resp. $\mathcal{D}^b(\mod A)$)
be the derived category (resp. bounded derived category) of complexes over $\Mod A$ (resp. $\mod A$).
Usually, we
just write $\mathcal{D} A$ (resp. $\mathcal{D}^b(A)$) instead of $ \mathcal{D}(\Mod A)$ (resp. $ \mathcal{D}^b(\mod A)$).

Up to isomorphism, the objects in $K^{b}(\proj A)$ are
precisely all the compact objects in $\mathcal{D} A$. For
convenience, we do not distinguish $K^{b}(\proj A)$ from the {\it
perfect derived category} $\mathcal{D}_{\per}(A)$ of $A$, i.e., the
full triangulated subcategory of $\mathcal{D} A$ consisting of all
compact objects, which will not cause any confusion. Moreover, we
also do not distinguish $K^b(\inj A)$ (resp.
$\mathcal{D}^b(A)$) from its essential image under the
canonical embedding into $\mathcal{D} A$.

Recall that an algebra $A$ is said to be {\it Gorenstein} if
$\id _AA <\infty$ and $\id _{A^{op}}A<\infty$.
The {\bf Gorenstein symmetry conjecture}
states that $\id _AA <\infty$ if and only if $\id _{A^{op}}A<\infty$.
This conjecture is listed in Auslander-Reiten-Smal{\o}'s book
\cite[p.410, Conjecture (13)]{ARS95}, and it closely connects with
other homological conjectures. For example, it is known that the
finitistic dimension conjecture implies the Gorenstein symmetry
conjecture. But so far all these conjectures are still open.

Following \cite{Buch87, Orl04}, the {\it singularity category} of $A$ is the
Verdier quotient $D_{sg}(A) = \mathcal{D}^b(A)/K^{b}(\proj A)$.
Recall that an $A$-module $M$ is called
{\it Gorenstein projective} if there is an
exact sequence $$\xymatrix{P^\bullet = \cdots \ar[r]& P^{-1}
\ar[r]^{d^{-1}} & P^0 \ar[r]^{d^{0}} & P^1
\ar[r] & \cdots} $$
of $\proj A$ with $M= \Ker d^0$ such that $\Hom _A(P^\bullet, Q)$ is exact for
every $Q \in \proj A$. Denote by $\Gproj A$
the subcategory of $\mod A$ consisting of Gorenstein projective modules.
It is well known that $\Gproj A$
is a Frobenius category, and hence its stable category $\underline{\Gproj} A$ is a triangulated category.
Moreover, there is a canonical triangle functor $F: \underline{\Gproj} A \rightarrow D_{sg}(A)$
sending a Gorenstein projective object to the corresponding stalk complex concentrated in degree zero.

\begin{definition}{\rm
The Verdier quotient $D_{def}(A):=D_{sg}(A)/ \Im F$
is called the Gorenstein defect
category of $A$.}
\end{definition}

For an algebra $A$, we define
${}^{\perp}A:=\{X\in\mod A|\Ext_A^i(X, A)=0\mbox{ for all }i>0\}$,
and denote $\underline{^\bot A}$ the stable category of $^\bot A$
modulo finitely generated projective $A$-modules. According to \cite[Theorem 2.12]{BM94}, $\underline{^\bot A}$
is a left triangulated
category with the standard left triangulated structure, and
it is clear that there is a embedding functor
$\underline{^\bot A} \hookrightarrow D_{sg}(A)$ of left triangulated
categories.

In an attempt to prove Nakayama
conjecture,  Auslander and Reiten \cite{AR75}
proposed the following conjecture: a finitely generated
module $M$ is projective if
$\Ext ^i _A(M, M\oplus A) = 0$, for any $i \geq 1$.
This conjecture is called {\bf Auslander-Reiten conjecture},
and it is true for
several classes of algebras, such as algebras of finite representation type,
syzygy-finite algebras, symmetric biserial algebras,
algebras with radical square zero and local algebras with radical cube zero \cite{AR75, Xu2013a, Xu2015}.

As a special case of Auslander-Reiten conjecture, Luo and Huang \cite{LH08}
proposed the {\bf Gorenstein projective conjecture}: a finitely generated Gorenstein projective
module $M$ is projective if
 $\Ext ^i _A(M,M)
= 0$, for any $i \geq 1$.
The Auslander-Reiten conjecture
and the Gorenstein projective conjecture coincide when $A$ is a
Gorenstein algebra, but it seems not true in general.
Moreover, the
Gorenstein projective conjecture is proved for
CM-finite algebras \cite{Zhang12}. For more development of this conjecture we refer to
\cite{LJ16}.

Let $A$ and $B$ be two
algebras
and $F : \mathcal{D} A\rightarrow \mathcal{D} B$ be a triangle functor. We
say that $F$ {\it restricts} to $K^b (\proj )$ (resp.
$\mathcal{D}^b(\mod)$, $K^b(\inj )$) if $F$ sends $K^b (\proj A)$ (resp.
$\mathcal{D}^b(\mod A)$, $K^b(\inj A)$) to $K^b (\proj B)$ (resp.
$\mathcal{D}^b(\mod B)$, $K^b(\inj B)$).

\section{Eventually homological isomorphisms and Gorenstein  projective modules}
\indent\indent In this section, we will compare the singularity categories,
Gorenstein defect categories and the stable categories of Gorenstein projective modules
between two algebras linked by an eventually homological isomorphism. Moreover, we
consider the Auslander-Reiten conjecture
and the Gorenstein symmetry conjecture for these algebras.

Recall that
a functor $F:\mod A\rightarrow \mod B$
is called an {\it eventually homological isomorphism}
if there is an integer $t$ such that
for every $j > t$, there is an
isomorphism $\Ext ^j _{\mathcal{B}}(X, Y ) \cong \Ext ^j _{\mathcal{C}}(FX, FY )$
for all objects $X, Y\in \mathcal{B}$.
Given the mallest such $t$, we call the functor a $t$-{\it eventually homological isomorphism}.
The following two theorems from \cite{PSS14} will be used frequently.

\begin{theorem}\label{theorem-ehi}{\rm (\cite[Theorem 4.3 (ii)]{PSS14})}
Let $F:\mod A\rightarrow \mod B$ be a $t$-eventually homological isomorphism
which is essentially surjective.

{\rm (i)} For every $X\in \mod A$, we have $\pd _BF(X)\leq \sup \{t, \pd _AX  \}$
and $\id _BF(X)\leq \sup \{t, \id _AX  \}$;

{\rm (ii)} For any $Y\in \mod B$, assume $F(Y')\cong Y$
for some $Y'\in \mod A$. Then we have $\pd _AY'\leq \sup \{t, \pd _BY \}$
and $\id _AY'\leq \sup \{t, \id _BY \}$.
\end{theorem}

\begin{theorem}\label{theorem-Gor}{\rm (\cite[Theorem 4.3 (v)]{PSS14})}
Let $F:\mod A\rightarrow \mod B$ be an eventually homological isomorphism
which is essentially surjective.
Then $A$ is
Gorenstein if and only if so is $B$.
\end{theorem}

Now we will reduce Gorenstein symmetry
conjecture by eventually homological isomorphisms.
\begin{theorem}\label{theorem-GSC}
Let $F:\mod A\rightarrow \mod B$ be an eventually homological isomorphism
which is essentially surjective.
Then $A$ satisfies the Gorenstein symmetry
conjecture if and only
if so does $B$.
\end{theorem}
\begin{proof}
Assume that $A$ satisfies the Gorenstein symmetry conjecture. If $\id _BB<\infty$,
then $K^b(\proj B)\subseteq K^b(\inj B)$ and by Theorem~\ref{theorem-ehi}, we have
$F(A)\in K^b(\proj B)\subseteq K^b(\inj B)$, that is, $\id _BF(A)<\infty$.
Using Theorem~\ref{theorem-ehi} again we have $\id _AA<\infty$.
Since $A$
satisfies the Gorenstein symmetry conjecture, we obtain that $A$ is Gorenstein.
By Theorem~\ref{theorem-Gor}, $B$
is Gorenstein and thus $\id _{B^{op}}B <\infty$.
Conversely, if $\id _{B^{op}}B <\infty$, then $\pd _{B}DB <\infty$. Therefore,
$K^b(\inj B)\subseteq K^b(\proj B)$ and by Theorem~\ref{theorem-ehi}, we have
$F(DA)\in K^b(\inj B)\subseteq K^b(\proj B)$, that is, $\pd _BF(DA)<\infty$.
Using Theorem~\ref{theorem-ehi} again we have $\pd _ADA<\infty$, that is, $\id _{A^{op}}A <\infty$.
Since $A$
satisfies the Gorenstein symmetry conjecture, we obtain that $A$ is Gorenstein.
By Theorem~\ref{theorem-Gor}, $B$
is Gorenstein and thus $\id _BB<\infty$.

Now assume that $B$ satisfies the Gorenstein symmetry conjecture. If $\id _AA<\infty$,
then $K^b(\proj A)\subseteq K^b(\inj A)$. Let
$B'\in \mod A$ such that $F(B')\cong B$. Then it follows from Theorem~\ref{theorem-ehi}
that $B'\in K^b(\proj A)\subseteq K^b(\inj A)$, that is, $\id _AB'<\infty$.
Using Theorem~\ref{theorem-ehi} again we have $\id _BB<\infty$.
Since $B$
satisfies the Gorenstein symmetry conjecture, we obtain that $B$ is Gorenstein.
By Theorem~\ref{theorem-Gor}, $A$
is Gorenstein and thus $\id _{A^{op}}A <\infty$.
Conversely, if $\id _{A^{op}}A <\infty$, then the statement
$\id _AA<\infty$ can be proved in a similar way.
\end{proof}

Let $A$ be an algebra and $e$ be an idempotent of $A$.
Then $A$ and $eAe$ are singularly equivalent
if $eA\otimes _A-: \mod A\rightarrow
\mod eAe$ is an eventually homological isomorphism, see \cite[Main Theorem]{PSS14}.
On the other hand,
the arrow removal operation
yields an eventually homological isomorphism
which induces a singular equivalence, see \cite[Main Theorem]{EPS21}.
Now we will unify these two results by showing that
two algebras linked by
a certain eventually homological isomorphism are always singularly equivalent.

\begin{theorem}\label{theorem-sing-equiv}
Let $F:\mod A\rightarrow \mod B$ be a $t$-eventually homological isomorphism
which is essentially surjective. Assume that $F$ admits a left adjoint
$H$ and a right adjoint $G$. Then $H$ and $F$ induce a singular equivalence
between $A$ and $B$.
\end{theorem}
\begin{proof}
Since $(H,F,G)$ is an adjoint triple, it follows that $H$ is right exact,
$G$ is left exact and $F$ is exact. Therefore, these derived functors
give rise to an adjoint triple $(LH,F,RG)$ between
$\mathcal{D}A$ and $\mathcal{D}B$.
Moreover, the exactness of $F$ implies that $F$ restricts to $\mathcal{D}^b(\mod )$,
and then $LH$ restricts to $K^b(\proj)$ by \cite[Lemma 2.7]{AKLY17}.
It follows from Theorem~\ref{theorem-ehi} that
$F$ restricts to $K^b(\proj)$ and $K^b(\inj)$, and then $LH$ restricts to
$\mathcal{D}^b(\mod )$ by \cite[Lemma 1]{QH16}.
Therefore, the functors $LH$ and $F$ induce an adjoint pair between
$D_{sg}(A)$ and $D_{sg}(B)$, see \cite[Lemma 1.2]{Orl04}.
Now we claim $F:D_{sg}(A)\rightarrow D_{sg}(B)$ is fully faithful and dense, and then
$$\xymatrix{ D_{sg}(B)\ar@<+1ex>[rr]^{LH}
&&D_{sg}(A)\ar@<+1ex>[ll]^{F}}$$ is a mutually inverse equivalence.

For any $X^\bullet \in D_{sg}(B)$, there exists some $n\in \mathbb{Z}$
and $X\in \mod B$ such that $X^\bullet \cong X[n]$ in $D_{sg}(B)$,
see \cite[Lemma 2.1]{Chen11}. Since $F:\mod A\rightarrow \mod B$ is essentially surjective,
we may assume that $X\cong F(X')$ for some $X'\in \mod A$.
Hence, we have $X^\bullet \cong F(X')[n]\cong F(X'[n])$ in $D_{sg}(B)$,
and thus $F:D_{sg}(A)\rightarrow D_{sg}(B)$ is dense.

For any $X^\bullet$,$Y^\bullet \in D_{sg}(A)$,
there exist $m$, $n\in \mathbb{Z}$
and $X$, $Y\in \mod A$ such that $X^\bullet \cong X[m]$ and
$Y^\bullet \cong Y[n]$ in $D_{sg}(A)$, see \cite[Lemma 2.1]{Chen11}.
Let $\eta$ be the counit of the adjoint pair
$(LH,F)$ between $\mathcal{D}^b(B)$ and $\mathcal{D}^b(A)$.
Then there is a canonical triangle $$ LHFX\rightarrow
X \rightarrow \Cone (\eta _X)\rightarrow $$ in $\mathcal{D}^b(A)$.
Let $S$ be a simple $A$-module.
Applying the functor $\Hom _{\mathcal{D}A}(-, S[i]) $,
we get an exact sequence
$$\Hom _{\mathcal{D}A}(\Cone (\eta _X), S[i]) \rightarrow
\Hom _{\mathcal{D}A}(X, S[i]) \rightarrow \Hom _{\mathcal{D}A}(LHFX, S[i]) \rightarrow .$$
For any $i>t$, we have isomorphisms
\begin{align*}
	\Hom _{\mathcal{D}A}(LHFX, S[i]) &\cong \Hom _{\mathcal{D}B}(FX, FS[i])\\
	&\cong \Ext _B^i(FX, FS) \\
	&\cong \Ext _A^i(X, S) \\
    & \cong \Hom _{\mathcal{D}A}(X, S[i]),
	\end{align*}
where the first isomorphism follows by adjunction, and the third one
is the definition of $t$-eventually homological isomorphism.
Therefore, we obtain
$\Hom _{\mathcal{D}A}(\Cone (\eta _X), S[i])\cong 0 $, for any $i>t+1$.
Now we claim that $\Cone (\eta _X) \in K^b(\proj A)$.
Since $\Cone (\eta _X)\in \mathcal{D}^b(A)$, $\Cone (\eta _X)$ is quasi-isomorphic to
a minimal right bounded complex
of finitely generated projective
$A$-modules. If this complex is not bounded, then some indecomposable
projective A-module with simple top $S$ occurs infinitely many times.
It follows that there are nonzero morphisms from  this complex to infinitely many
positive shifts of $S$, that is,
$\Hom _{\mathcal{D}A}(\Cone (\eta _X), S[i]) \neq 0$ for infinite many $i\in \mathbb{Z}^+$.
But this is a contradiction. Therefore, $\Cone (\eta _X) \in K^b(\proj A)$
and thus $LHFX\cong X$ in $D_{sg}(A)$. As a result, we have isomorphisms
\begin{align*}
	\Hom _{D_{sg}(A)}(X^\bullet, Y^\bullet) &\cong \Hom _{D_{sg}(A)}(X[m], Y[n])\\
	&\cong  \Hom _{D_{sg}(A)}(LHFX[m], Y[n]) \\
	&\cong  \Hom _{D_{sg}(B)}(FX[m], FY[n]) \\
    & \cong \Hom _{D_{sg}(B)}(FX^\bullet, FY^\bullet),
	\end{align*}
and then $F:D_{sg}(A)\rightarrow D_{sg}(B)$ is fully faithful.
\end{proof}

Now we will investigate Gorenstein defect categories
in the setting of eventually homological isomorphisms.
We start with the following result.
\begin{lemma}\label{lemma-perp}
Let $F:\mod A\rightarrow \mod B$ be a $t$-eventually homological isomorphism
which is essentially surjective.
Then $\Omega ^t(FX)\in {}^{\perp} B$ for any $X\in {}^{\perp}A$.
\end{lemma}
\begin{proof}
By definition, there exists some $B'\in \mod A$ such that $F(B')\cong B$,
and it follows from Theorem~\ref{theorem-ehi} that
$\pd _AB'\leq t$.
For any $i>0$, consider the isomorphisms
$$\Ext_B^i(\Omega ^t(FX),B)\cong \Ext_B^{i+t}(FX,B)\cong
\Ext_B^{i+t}(FX,F(B'))\cong \Ext_A^{i+t}(X,B'),$$
where the last equation holds because $F$ is a $t$-eventually homological isomorphism.
Since $X\in {}^{\perp}A$, we can use the projective resolution
of $B'$ to do dimension shifting, that is,
$\Ext_A^{i+t}(X,B')\cong \Ext_A^{i+2t}(X,\Omega ^tB')$.
Since $\pd _AB'\leq t$, we have that $\Omega ^tB'\in \proj A$
and thus $\Ext_A^{i+2t}(X,\Omega ^tB')\cong 0$.
Above all, we get $\Ext_B^i(\Omega ^t(FX),B)\cong0$ for any $i>0$.
\end{proof}

Denote by $D^b(A)_{fGd}$ the full subcategory of
$D^b(A)$ formed by those complexes quasi-isomorphic to bounded complex of Gorenstein projective objects.
Here, the definition of $D^b(A)_{fGd}$ agrees with that in \cite{Kato02},
where the objects in $D^b(A)_{fGd}$ are called complexes of finite Gorenstein projective dimension,
see \cite[Definition 2.7 and Proposition 2.10]{Kato02}.
 Moreover, $D^b(A)_{fGd}$ is a thick subcategory
of $D^b(A)$ generated by all the Gorenstein projective modules, that is,
$D^b(A)_{fGd}=\thick (\Gproj A)$, see \cite[Theorem 2.7]{LHZ22}.
The following equivalence
$$ \underline{\Gproj }A \cong D^b(A)_{fGd}/K^b(\proj A)$$
is well known, see \cite[Theorem 4.4.1]{Buch87}, \cite[Lemma 4.1]{CR20}
or \cite[Theorem 4]{PZ15} for examples.

The following lemma is an alternative description of Gorenstein projective objects.

\begin{lemma}\label{lemma-Gproj}{\rm (\cite[Lemma 5.1]{HP17})}
An object $X\in \mod A$ is Gorenstein projective if and only if
there are short exact sequences $0\rightarrow X^i\rightarrow P^{i+1}\rightarrow X^{i+1}\rightarrow 0$
in $\mod A$ with $ P^{i}$ projective and $X^i\in {}^{\perp}A$ for all $i \in \mathbb{Z}$ such that $X^0 = X$.
\end{lemma}

Next we show that certain eventually homological isomorphisms
preserve complexes of finite Gorenstein projective dimension.
\begin{lemma}\label{lemma-fGd}
Let $F:\mod A\rightarrow \mod B$ be a $t$-eventually homological isomorphism
which is exact and essentially surjective.
Then $\Omega ^t(FX)\in \Gproj B$ for any $X\in \Gproj A$, and
$F$ induces a triangle functor from $D^b(A)_{fGd}$ to $D^b(B)_{fGd}$.
\end{lemma}
\begin{proof}
Since $F:\mod A\rightarrow \mod B$ is exact, we have an induced functor
$F: D^b(A)\rightarrow D^b(B)$. Let $X\in \mod A$ be a Gorenstein projective module.
By Lemma~\ref{lemma-Gproj}, there are short exact sequences
$0\rightarrow X^i\rightarrow P^{i+1}\rightarrow X^{i+1}\rightarrow 0$
in $\mod A$ with $ P^{i}$ projective and $X^i\in {}^{\perp}A$ for all $i \in \mathbb{Z}$ such that $X^0 = X$.
Since $F$ is exact, the sequences $$0\rightarrow FX^i\rightarrow FP^{i+1}\rightarrow FX^{i+1}\rightarrow 0$$
are exact, and these lead to exact sequences $$0\rightarrow \Omega ^t(FX^i)\rightarrow
\Omega ^t(FP^{i+1})\oplus Q^{i+1}\rightarrow \Omega ^t(FX^{i+1})\rightarrow 0,$$
where $Q^{i+1}\in \proj B$. Since $X^i\in {}^{\perp}A$, it follows from
Lemma~\ref{lemma-perp} that $\Omega ^t(FX^i)\in {}^{\perp}B$ for all $i \in \mathbb{Z}$,
and by Theorem~\ref{theorem-ehi},
$\Omega ^t(FP^{i+1})\in \proj B$ for all $i \in \mathbb{Z}$. Now Lemma~\ref{lemma-Gproj}
shows that
$\Omega ^t(FX)\in \Gproj B$, and then $FX\in \thick (\Gproj B)$.
Above all, we conclude that $FX \in D^b(B)_{fGd}$ for any $X\in \Gproj A$. Therefore,
$$F (D^b(A)_{fGd})=F(\thick (\Gproj A))\subseteq \thick F(\Gproj A)\subseteq D^b(B)_{fGd}.$$
\end{proof}

Following \cite{HP17},
a triangle functor $F: D^b(A)\rightarrow D^b(B)$ is said to be
{\it non-negative} if $F$ satisfies the following conditions:
(1) $F(X)$ is isomorphic to a complex with zero homology in all
negative degrees, for all $X \in \mod A$;
(2) $F(A)$ is isomorphic to a complex in $K^b(\proj B)$
with zero terms in all negative degrees.
Now we are ready to compare the Gorenstein defect categories and the stable categories of
Gorenstein projective modules between two algebras linked by an eventually homological isomorphism.
\begin{theorem}\label{theorem-Gproj-equiv}
Let $F:\mod A\rightarrow \mod B$ be a $t$-eventually homological isomorphism
which is essentially surjective. Assume that $F$ admits a left adjoint
$H$ and a right adjoint $G$.
Then $F$
induces triangle equivalences
$ \underline{\Gproj }A\cong \underline{\Gproj }B$ and
$\ D_{def}(A)\cong D_{def}(B)$.
\end{theorem}
\begin{proof}
Since $(H,F)$ is an adjoint pair, we infer that
$H$ preserves direct sums and $H$ is right exact. By Watt's theorem,
$H$ is isomorphic to $H(B)\otimes _B-:\mod B\rightarrow \mod A$,
where the right $B$-module structure of $H(B)$ is given by
$B\cong \Hom_B(B,B)\rightarrow \Hom_A(H(B),H(B))$. Now
consider the derived functor $LH\cong H(B)\otimes _B^L-:\mathcal{D} B\rightarrow \mathcal{D} A$.
By the proof of Theorem~\ref{theorem-sing-equiv}, we have that
$LH$ and $F$ restrict to both $\mathcal{D}^b(\mod)$ and $K^b(\proj )$,
and using \cite[Lemma 2.8]{AKLY17},
we get that $LH$ has a left adjoint which restricts to $K^b(\proj )$.
It follows from \cite[Lemma 3.4]{CHQW20} that $LH$ restricts to
a non-negative functor from $\mathcal{D}^b(B)$ to $\mathcal{D}^b(A)$, up to shifts.
By \cite[Proposition 5.2]{HP17},
the stable functor $\overline{LH}$ preserves Gorenstein projective
modules. According to \cite[Section 4.2]{HP17},
each $X\in \mod B$ yields a triangle $$P^\bullet _X\rightarrow LH(X)\rightarrow
\overline{LH}(X)\rightarrow$$ in $\mathcal{D}A$ with $P^\bullet _X\in K^b(\proj A)$.
Therefore, $LHX \in D^b(A)_{fGd}$ for any $X\in \Gproj B$, and then
$LH$ sends the objects of $D^b(B)_{fGd}$ to $D^b(A)_{fGd}$.
In view of Lemma~\ref{lemma-fGd}, $LH$ and $F$ induce an adjoint pair between
$D^b(B)_{fGd}/ K^b(\proj B)$ and $D^b(A)_{fGd}/ K^b(\proj A)$, see \cite[Lemma 1.2]{Orl04}.
Thanks to the equivalence
$ \underline{\Gproj }A \cong D^b(A)_{fGd}/K^b(\proj A)$, we obtain
an adjoint pair
$$\xymatrix{ \underline{\Gproj} B\ar@<+1ex>[rr]^{\widetilde{LH}}
&&\underline{\Gproj} A\ar@<+1ex>[ll]^{\widetilde{F}}},$$
and combining Theorem~\ref{theorem-sing-equiv}, we have the following exact commutative diagram
$$\xymatrix{0 \ar[r] &\underline{\Gproj} A \ar[r] \ar@<-1ex>[d]_ {\widetilde{F}} & D_{sg}(A)
\ar@<-1ex>[d]_F \ar[r] & D_{def}(A) \ar[r] & 0
\\ 0 \ar[r] &\underline{\Gproj} B \ar[r] \ar@<-1ex>[u]_{\widetilde{LH}} & D_{sg}(B) \ar[r] \ar@<-1ex>[u]_{LH} & D_{def}(B) \ar[r] & 0 }, $$
where the vertical functors between $D_{sg}(A)$ and $D_{sg}(B)$ are equivalences.
Hence, $LH$ and $F$ induce an equivalence between $\underline{\Gproj} A$ and $\underline{\Gproj} B$,
and also, there is an equivalence between $D_{def}(A)$ and $D_{def}(B)$.
\end{proof}

Now let's turn to the invariance of the Auslander-Reiten conjecture
(resp. Gorenstein projective conjecture) under  eventually
homological isomorphisms.
\begin{theorem}\label{theorem-ARC}
Assume that $F:\mod A\rightarrow \mod B$ satisfies all the conditions
in Theorem~\ref{theorem-Gproj-equiv}.
Then $A$ satisfies the Auslander-Reiten conjecture
(resp. Gorenstein projective conjecture) if and only
if so does $B$.
\end{theorem}
\begin{proof}
By the proof of Theorem~\ref{theorem-Gproj-equiv}, the functor $LH$ restricts to
a non-negative functor from $\mathcal{D}^b(B)$ to $\mathcal{D}^b(A)$.
Moreover, $LH$ admits a right adjoint $F$ which preserves $K^b(\proj)$.
Therefore, it follows from \cite[Proposition 4.8 and Proposition 5.2]{HP17} that
there are two commutative diagrams
$$\xymatrix@!=1pc{ \underline{^\bot B} \ar @{^{(}->}[d]
\ar[rr]^{\overline{LH}} &&\underline{^\bot A} \ar @{^{(}->}[d] \\
D_{sg}(B)
\ar[rr]^{LH} &&D_{sg}(A)
} \begin{array}{c}
\\  \\ \mbox{and} \\ \end{array}
\xymatrix@!=1pc{ \underline{\Gproj} B \ar @{^{(}->}[d]
\ar[rr]^{\overline{LH}} &&\underline{\Gproj} A \ar @{^{(}->}[d] \\
D_{sg}(B)
\ar[rr]^{LH} &&D_{sg}(A).
}
$$
It is clear that the embedding $\underline{^\bot A} \hookrightarrow D_{sg}(A)$
induces an isomorphism $$\Ext _A^i(X,Y)\cong \Hom _{D_{sg}(A)}(X,Y[i])$$
for each $X, Y \in {}^{\perp}A$ and for each $i > 0$.
Moreover, it follows from Theorem~\ref{theorem-sing-equiv} that $LH: D_{sg}(B) \rightarrow D_{sg}(A)$
is an equivalence.
Hence, using the same judgment as \cite[Lemma 3.3]{CHQW20},
we can prove that the Auslander-Reiten conjecture
(resp. Gorenstein projective conjecture) holds for $B$ if it holds for $A$.

Note that the functor $F: \mod A\rightarrow \mod B$ is exact, and $\pd _BF(P)\leq t$
for any $P\in \proj A$. Then the functor $F[-t]: \mathcal{D}^b(A) \rightarrow \mathcal{D}^b(B)$
is non-negative and the stable functor $\overline{F[-t]}$ is isomorphic to
$\Omega ^tF$.
Combining \cite[Proposition 4.8]{HP17} with Lemma~\ref{lemma-perp} and Lemma~\ref{lemma-fGd},
we have two commutative diagrams
$$\xymatrix@!=1pc{ \underline{^\bot A} \ar @{^{(}->}[d]
\ar[rr]^{\overline{F[-t]}} &&\underline{^\bot B} \ar @{^{(}->}[d] \\
D_{sg}(A)
\ar[rr]^{F[-t]} &&D_{sg}(B)
} \begin{array}{c}
\\  \\ \mbox{and} \\ \end{array}
\xymatrix@!=1pc{ \underline{\Gproj} A \ar @{^{(}->}[d]
\ar[rr]^{\overline{F[-t]}} &&\underline{\Gproj} B \ar @{^{(}->}[d] \\
D_{sg}(A)
\ar[rr]^{F[-t]} &&D_{sg}(B).
}
$$
By Theorem~\ref{theorem-sing-equiv}, the functor $F: D_{sg}(A) \rightarrow D_{sg}(B)$
is an equivalence, and so is $F[-t]: D_{sg}(A) \rightarrow D_{sg}(B)$. Therefore,
the Auslander-Reiten conjecture
(resp. Gorenstein projective conjecture) holds for $B$ implies that it holds for $A$.

\end{proof}

\section{Applications and examples}
\indent\indent
In this section, we will apply our main results to arrow removal and vertex removal.
This will produce new reduction techniques
for the study of Gorenstein defect categories, Gorenstein
symmetry conjecture, Auslander-Reiten conjecture and
Gorenstein projective conjecture.

Let $A$ be an admissible quotient $kQ/I$ of a path algebra $kQ$ over a field $k$.
Choose an arrow $\alpha$ in $Q$ which does not occur in a minimal generating set
of $I$ and define $B =A/\langle \overline{\alpha} \rangle$. The arrow removal operation, transferring homological properties
of $A$ to $B$,
was investigated in \cite{EPS21, GPS18} with respect to finitistic dimension,
Gorensteinness, singularity categories
and the Fg condition. Now we will consider other
homological invariants under this operation.

\begin{corollary}\label{cor-arrow}
Keep the above notations and assumptions.
Then
$ \underline{\Gproj }A\cong \underline{\Gproj }B$ and
$D_{def}(A)\cong D_{def}(B).$ Moreover, $A$ satisfies the Gorenstein
symmetry conjecture (resp. Auslander-Reiten conjecture,
Gorenstein projective conjecture) if and only if so does $B$.
\end{corollary}
\begin{proof}
By \cite[Proposition 4.6]{GPS18}, there is a functor
$e: \mod A\rightarrow
\mod B$ which is essentially surjective and admits a left adjoint
and a right adjoint.
Further, it follows from \cite[Corollary 3.3]{EPS21} that $e: \mod A\rightarrow
\mod B$ is an eventually homological isomorphism. Now this
corollary follows from Theorem~\ref{theorem-GSC}, Theorem~\ref{theorem-Gproj-equiv} and Theorem~\ref{theorem-ARC}.
\end{proof}

Let A be an algebra and $e$ be an idempotent in $A$.
In \cite{PSS14}, the author proved that $eA\otimes _A-: \mod A\rightarrow
\mod eAe$ is an eventually homological isomorphism
if and only if $\pd _{eAe}eA< \infty $ and $\id _A(\frac {A/AeA}{\rad (A/AeA)})<\infty$ (or equivalently,
$\pd _{(eAe)^{op}}Ae< \infty $ and $\pd _A(\frac {A/AeA}{\rad (A/AeA)})<\infty$), and
they compared the algebras $A$ and $eAe$ with respect to Gorensteinness, singularity categories
and the Fg condition under these conditions. Now
we will investigate more homological invariants between $A$ and $eAe$.

\begin{corollary}\label{cor-idem}
Assume that
$\pd _{eAe}eA< \infty $ and $\id _A(\frac {A/AeA}{\rad (A/AeA)})<\infty$
(or equivalently, $\pd _{(eAe)^{op}}Ae< \infty $ and $\pd _A(\frac {A/AeA}{\rad (A/AeA)})<\infty$).
Then the functor $eA\otimes _A-: \mod A\rightarrow
\mod eAe$ induces triangle equivalences
$ \underline{\Gproj }A\cong \underline{\Gproj }(eAe)$ and
$D_{def}(A)\cong D_{def}(eAe).$ Moreover, $A$ satisfies the Gorenstein
symmetry conjecture (resp. Auslander-Reiten conjecture,
Gorenstein projective conjecture) if and only if so does $eAe$ .
\end{corollary}
\begin{proof}
Clearly, any idempotent element $e$ induces a recollement
between $\mod A/AeA$, $\mod A$ and $\mod eRe$. Therefore,
the functor $eA\otimes _A-: \mod A\rightarrow
\mod eAe$ is essentially surjective, and it admits a left adjoint
and a right adjoint.
Further, it follows from \cite[Main theorem]{PSS14} that $eA\otimes _A-$ is an eventually homological isomorphism. Now this
corollary follows from Theorem~\ref{theorem-GSC}, Theorem~\ref{theorem-Gproj-equiv} and Theorem~\ref{theorem-ARC}.
\end{proof}

Using \cite[Lemma 8.11, Lemma 8.9 and Proposition 8.7 ]{PSS14}, we have the following special
case of Corollary~\ref{cor-idem}.
\begin{corollary}\label{cor-vertex}
Let $A = kQ/I$ be a quotient of a path algebra $kQ$
over a field $k$. Choose
some vertices in $Q$ where no relations start and no relations end, and let $e$ be the
sum of idempotents corresponding to all vertices except these. Then
$ \underline{\Gproj }A\cong \underline{\Gproj }(eAe)$ and
$D_{def}(A)\cong D_{def}(eAe).$ Moreover, $A$ satisfies the Gorenstein
symmetry conjecture (resp. Auslander-Reiten conjecture,
Gorenstein projective conjecture) if and only if so does $eAe$ .
\end{corollary}

Let $A$ and $B$ be algebras, $M$ an $A$-$B$-bimodule and $T =
\left[\begin{array}{cc} A & M \\ 0 & B  \end{array}\right] $.
Let $e_A= \left[\begin{array}{cc} 1_A & 0 \\ 0 & 0  \end{array}\right]$ and
$e_B= \left[\begin{array}{cc} 0 & 0 \\ 0 & 1_B  \end{array}\right] $. Denote by
$S_{e_A}=e_AT\otimes _T-: \mod T\rightarrow
\mod e_ATe_A$ and $S_{e_B}=e_BT\otimes _T-: \mod T\rightarrow
\mod e_BTe_B$. From \cite{Z13}, $_AM_B$ is {\it compatible} if $M\otimes _B-$
sends every acyclic complex of projective $B$-modules
to acyclic complex, and $\Ext _A^i(G,M)\cong 0$ for any $G\in \Gproj A$
and $i>0$.
Assume that $_AM_B$ is compatible, then $S_{e_A}$
induces triangle equivalences $D_{sg}(T)\cong D_{sg}(A)$, $D_{def}(T)\cong D_{def}(A)$
and $\underline{\Gproj }T\cong \underline{\Gproj }A$
if and only if $\gl B<\infty$, $\pd _AM <\infty$ and $M\otimes _B-$ preserve modules of
finite Gorenstein projective dimension, see \cite[Theorem 4.4 (2)]{LHZ22}.
Now we will use Corollary~\ref{cor-idem} to simplify these conditions.

\begin{corollary}\label{cor-tri-alg-A}{\rm (Compare \cite[Theorem 4.4 (2)]{LHZ22})}
Let  $T =
\left[\begin{array}{cc} A & M \\ 0 & B  \end{array}\right] $
be a triangular matrix algebra. Then $S_{e_A}$ induces
triangle equivalences $D_{sg}(T)\cong D_{sg}(A)$, $D_{def}(T)\cong D_{def}(A)$
and $\underline{\Gproj }T\cong \underline{\Gproj }A$
if and only if $\gl B<\infty$ and $\pd _AM <\infty$.
\end{corollary}
\begin{proof}
Assume that $\gl B<\infty$ and $\pd _AM <\infty$. Then
it follows from \cite[Lemma 8.15]{PSS14} that $\pd _{e_ATe_A}e_AT< \infty $
and $\id _A(\frac {T/Te_AT}{\rad (T/Te_AT)})<\infty$.
According to \cite[Main theorem]{PSS14}, $S_{e_A}$ induces
a triangle equivalence $D_{sg}(T)\cong D_{sg}(A)$,
and by Corollary~\ref{cor-idem}, $S_{e_A}$ induces equivalences
$ \underline{\Gproj }T\cong \underline{\Gproj }A$ and
$D_{def}(T)\cong D_{def}(A).$ Conversely, assume $S_{e_A}$ induces
such triangle equivalences. Then it follows from \cite[Theorem 4.4 (2)]{LHZ22}
that $\gl B<\infty$ and $\pd _AM <\infty$.
Indeed, the compatibility of $M$ is not used
in the only if part of \cite[Theorem 4.4 (2)]{LHZ22}.
\end{proof}

Similarly, assume that $_AM_B$ is compatible. Then $S_{e_B}$
induces triangle equivalences $D_{sg}(T)\cong D_{sg}(B)$, $D_{def}(T)\cong D_{def}(B)$
and $\underline{\Gproj }T\cong \underline{\Gproj }B$
if and only if $\gl A<\infty$, see \cite[Theorem 4.6 (2)]{LHZ22}.
Now we will give a sufficient conditions without the compatibility of $M$.

\begin{corollary}\label{cor-tri-alg-B}{\rm (Compare \cite[Theorem 4.6 (2)]{LHZ22})}
Let  $T =
\left[\begin{array}{cc} A & M \\ 0 & B  \end{array}\right] $
be a triangular matrix algebra. Assume that $\gl A<\infty$ and $\pd _{B^{op}}M<\infty$,
then
$S_{e_B}$ induces
triangle equivalences $D_{sg}(T)\cong D_{sg}(B)$, $D_{def}(T)\cong D_{def}(B)$
and $\underline{\Gproj }T\cong \underline{\Gproj }B$.
\end{corollary}
\begin{proof}
Assume that $\gl A<\infty$ and $\pd _{B^{op}}M<\infty$. Then
it follows from \cite[Lemma 8.16]{PSS14} that $\pd _{e_BTe_B}e_BT< \infty $
and $\id _B(\frac {T/Te_BT}{\rad (T/Te_BT)})<\infty$.
Therefore, the statement follows from
\cite[Main theorem]{PSS14} and Corollary~\ref{cor-idem}.
\end{proof}

Now we will illustrate our results by three examples. In particular, the Gorenstein projective
modules of some non-monomial
algebras are described.

\begin{example}
{\rm Let $A$ be the $k$-algebra given by the following quiver
$$
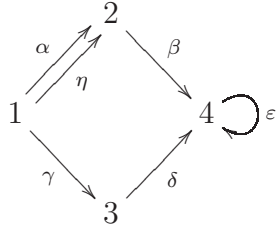
\xymatrix{ & 2 \ar[rd]^\beta   &  & \\
1  \ar@<+0.5ex>[ur]^\alpha \ar@<-0.5ex>[ur]_\eta \ar[rd]_\gamma &   &4 \ar@(ur,dr)^\varepsilon & \\
 & 3  \ar[ur]_\delta  &  &
}$$
with relations
$\{ \varepsilon ^2, \beta \eta, \beta \alpha -\delta \gamma\}$. We write the concatenation
of paths from right to left. Clearly, there is no relation starting and ending
at the vertices $2$ and $3$, and then it follows from Corollary~\ref{cor-vertex} that
$ \underline{\Gproj }A\cong \underline{\Gproj }B$ and
$D_{def}(A)\cong D_{def}(B)$, where $B$ is the algebra
$$\xymatrix{1 \ar[r]^\lambda & 4 \ar@(ur,dr)^\varepsilon}$$
with relations $\{\varepsilon ^2\}$.
Using Corollary~\ref{cor-vertex} again, we have that $\underline{\Gproj }B\cong \underline{\Gproj }C$ and
$D_{def}(B)\cong D_{def}(C)$, where $C$ is the algebra
$$\xymatrix{4 \ar@(ur,dr)^\varepsilon}$$
with relations $\{\varepsilon ^2\}$.
Since $C$ is selfinjective, we get $\underline{\Gproj }C\cong \underline{\mod }C\cong
\mod k$ and $D_{def}(C)\cong 0$. Above all, we conclude that
$\underline{\Gproj }A\cong
\mod k$ and $D_{def}(A)\cong 0$.
It is easy to check that the simple $A$-module corresponding to $4$
is Gorenstein projective, and then
all Gorenstein projective modules over $A$ are this simple module and projective modules.
}
\end{example}

\begin{example}
{\rm  This is Example 1 from \cite{Xi04} and Example 6.5 from \cite{GPS18}.
Let $A$ be the $k$-algebra given by the following quiver
$$
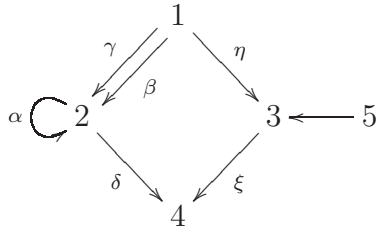
\xymatrix{ & 1 \ar@<+0.5ex>[ld]^\beta \ar@<-0.5ex>[ld]_\gamma \ar[rd]^\eta &  &  & \\
2 \ar@(ul,dl)_\alpha \ar[rd]_\delta &  & 3 \ar[ld]^\xi & 5 \ar[l]& \\
 & 4 &  &  &
}$$
with relations
$\{ \alpha ^3, \delta \alpha, \delta \beta, \xi \eta -\delta \gamma  \}$.
Clearly, there is no relation starting and ending
at the vertices $3$ and $5$, and it follows from Corollary~\ref{cor-vertex} that
$ \underline{\Gproj }A\cong \underline{\Gproj }B$ and
$D_{def}(A)\cong D_{def}(B)$, where $B$ is the algebra
$$\xymatrix{1 \ar@<+0.5ex>[r]^\gamma \ar@<-0.5ex>[r]_\beta & 2  \ar@(ul,ur)^\alpha
\ar[r]^\delta &4}$$
with relations $\{ \alpha ^3, \delta \alpha, \delta \beta \}$.
Now consider the algebra $B$. Let $e=e_2+e_4$ and let $S_i$ denote the simple $B$-module
associated to the vertex $i$. Then $\pd _{(eBe)^{op}}Be <\infty$
and $\pd _B(\frac {B/BeB}{\rad (B/BeB)})=\pd _B(S_1)<\infty$.
Hence, according to Corollary~\ref{cor-idem}, we have $\underline{\Gproj }B\cong \underline{\Gproj }(C)
$ and
$D_{def}(B)\cong D_{def}(C),$ where $C=eBe$ is the algebra
$$\xymatrix{2  \ar@(ul,ur)^\alpha
\ar[r]^\delta &4}$$
with relations $\{ \alpha ^3, \delta \alpha \}$. Since $C$ is a monomial algebra,
its Gorenstein-projective modules were described in \cite{CSZ18}. Indeed, $C$ is CM-free
and then $\underline{\Gproj }A\cong \underline{\Gproj }B \cong \underline{\Gproj }C\cong0$.
Therefore, all Gorenstein projective modules over $A$ are projective.
Now
consider the algebra $C$. Since $\pd S_4<\infty$ and $\pd _{e_2Ae_2}e_2A<\infty$,
it follows from \cite[Theorem 2.1]{Chen09} that $D_{sg}(C)\cong D_{sg}(k[x]/\langle x^3\rangle)$.
Then we get $D_{def}(A)\cong D_{def} (C)\cong  D_{sg}(C)\cong
\underline{\mod} (k[x]/\langle x^3\rangle)$.
Note that monomial algebras satisfy the Gorenstein
symmetry conjecture, Auslander-Reiten conjecture and
Gorenstein projective conjecture. Hence, we conclude that $A$
satisfies these conjectures by Corollary~\ref{cor-idem} and Corollary~\ref{cor-vertex}.
}
\end{example}

\begin{example}
{\rm This is Example 6.3 from \cite{GPS18}.
Let $A$ be the $k$-algebra given by the following quiver
$$
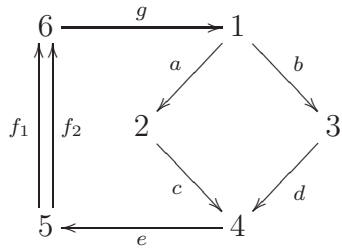
\xymatrix{6 \ar[rr]^g& & 1 \ar[ld]_a  \ar[rd]^b &  & \\
& 2 \ar[rd]_c & & 3 \ar[ld]^d& \\
 5 \ar@<+0.5ex>[uu]^{f_1} \ar@<-0.5ex>[uu]_{f_2}& & 4 \ar[ll]^e&  &
}$$
with relations
$\{ ca-db, f_1ec, ed, gf_1e, bgf_1, ag \}$. Then
$A$ is a representation infinite non-monomial
algebra and $\gl A=\infty$. However,
$A$ can be reduced by Corollary~\ref{cor-arrow}
since $f_2$ is not occurring in any relations.
Note that the algebra $A/\langle \overline{f_2} \rangle$
is of finite representation type, and then it satisfies the Gorenstein
symmetry conjecture, Auslander-Reiten conjecture and
Gorenstein projective conjecture. Hence, we conclude that $A$
satisfies these conjectures by Corollary~\ref{cor-arrow}.

}
\end{example}

\noindent {\footnotesize {\bf ACKNOWLEDGMENT.}
This work is supported by the National Natural Science
Foundation of China (Grant No.12061060) and the
Scientific and Technological Innovation Team of Yunnan Province, China (Grant
No. 2020CXTD25).}

\end{document}